\documentclass[12pt]{amsart}
\usepackage{amssymb,amsmath}

\usepackage{geometry}    
\usepackage{mathrsfs}

\usepackage{vmargin}
\setmargrb{1in}{1in}{1in}{1in}

\newtheorem{theorem}{Theorem}[section]
\newtheorem{lemma}[theorem]{Lemma}

\newtheorem{corollary}[theorem]{Corollary}
\theoremstyle{definition}

\newtheorem{remark}{Remark}

\begin{document}

\title[Best versus uniform approximation]{On the discrepancy between best and uniform approximation}

\author{Johannes Schleischitz}
                            
\thanks{Supported by the Schr\"odinger Scholarship J 3824 of the Austrian Science Fund (FWF).\\
University of Ottawa, Department of Mathematics and Statistics, King Edward 585, ON K1N 6N5 \\
	johannes.schleischitz@univie.ac.at}

\begin{abstract}
For $\zeta$ a transcendental real number, we consider the classical Diophantine exponents $w_{n}(\zeta)$
and $\widehat{w}_{n}(\zeta)$. They
measure how small $\vert P(\zeta)\vert$ can be for an integer polynomial $P$
of degree at most $n$ and naive height bounded by $X$, for arbitrarily large and all large $X$, respectively.
The discrepancy between the
exponents $w_{n}(\zeta)$ and $\widehat{w}_{n}(\zeta)$ has attracted interest recently. Studying parametric geometry
of numbers, W. Schmidt and L. Summerer were the first to refine the trivial inequality $w_{n}(\zeta)\geq \widehat{w}_{n}(\zeta)$.
Y. Bugeaud and the author found another estimation 
provided that the condition $w_{n}(\zeta)>w_{n-1}(\zeta)$ holds.
In this paper we establish an unconditional version of the latter result, which can be regarded as a proper extension. 
Unfortunately, the new contribution involves an additional exponent and is of interest only in certain cases.   
\end{abstract}

\maketitle

{\footnotesize{Math subject classification: 11J13, 11J25, 11J82 \\
key words: Diophantine inequalities, exponents of Diophantine approximation, U-numbers}}

\section{Introduction} \label{sek1}

Let $n$ be a positive integer and $\zeta$ be a transcendental real number. 
For a polynomial $P$ as usual let $H(P)$ denote its height, 
which is the maximum modulus among the coefficients of $P$. 
We want to investigate relations between the two classical exponents of Diophantine approximation
$w_{n}(\zeta)$ and $\widehat{w}_{n}(\zeta)$ introduced below.
Define $w_{n}(\zeta)$ as the supremum 
of $w\in{\mathbb{R}}$ such that
\begin{equation}  \label{eq:w}
H(P) \leq X, \qquad  0<\vert P(\zeta)\vert \leq X^{-w},  
\end{equation}
has a solution $P\in\mathbb{Z}[T]$ of degree at most $n$ for arbitrarily large values of $X$.
Similarly, let $\widehat{w}_{n}(\zeta)$ be the supremum of $w$ such that \eqref{eq:w}
has a solution $P\in\mathbb{Z}[T]$ of degree at most $n$ for {\em all} large $X$.
The interest of the exponent in \eqref{eq:w} and the derived exponents
arises partly from the relation to approximation to a real number
by algebraic numbers of bounded degree. Indeed, when $\alpha$ is an algebraic number very close to $\zeta$, 
then the evaluation $P_{\alpha}(\zeta)$ is also very small by absolute value, for $P_{\alpha}$ the irreducible 
minimal polynomial of $\alpha$ over $\mathbb{Z}$.
More precisely $\vert P_{\alpha}(\zeta)\vert \leq C(n,\zeta) H(P_{\alpha})\vert \zeta-\alpha\vert$ for a constant 
that depends only on $\zeta$ and the degree $n$ of $\alpha$.
The converse is a delicate problem, at least for certain 
numbers $\zeta$, related to the famous problem
of Wirsing posed in~\cite{wirsing}.
We do not further discuss this issue here and only affirm that results involving the exponents $w_{n}, \widehat{w}_{n}$
typically imply comparable results concerning approximation by algebraic numbers in an obvious way.
For any real number $\zeta$, our exponents clearly 
satisfy the relations
\begin{equation} \label{eq:wmonos}
w_{1}(\zeta)\leq w_{2}(\zeta)\leq \cdots, 
\qquad \widehat{w}_{1}(\zeta)\leq \widehat{w}_{2}(\zeta)\leq \cdots. 
\end{equation}
Dirichlet's box principle further implies
\begin{equation} \label{eq:wmono}
w_{n}(\zeta)\geq \widehat{w}_{n}(\zeta)\geq n.
\end{equation}
The value $w_{n}(\zeta)$ can be infinity. In this case $\zeta$ is called a $U$-number, more precisely
$\zeta$ is called $U_{n}$-number if $n$ is the smallest such index. 
The existence of $U_{n}$-numbers for any $n\geq 1$ was first proved by LeVeque~\cite{leveque}. 
On the other hand,
the quantities $\widehat{w}_{n}(\zeta)$ can be effectively bounded.
For $n=1$, it is not hard to see that we always 
have $\widehat{w}_{1}(\zeta)=1$, see~\cite{khin}.
For $n=2$, Davenport and Schmidt~\cite{davsh} showed 
\begin{equation} \label{eq:najo}
\widehat{w}_{2}(\zeta)\leq \frac{3+\sqrt{5}}{2}= 2.6180\ldots.
\end{equation}
Roy~\cite{royyy} proved that for certain numbers 
he called {\em extremal numbers} there is equality, so \eqref{eq:najo} is sharp. For an overview of the results
on the values $\widehat{w}_{2}(\zeta)$ attained for real $\zeta$, see~\cite[Section~2.4]{bdraft}. 
For $n\geq 3$, little is known about the 
exponents $\widehat{w}_{n}(\zeta)$. 
The supremum of the values $\widehat{w}_{n}(\zeta)$ over all real real numbers $\zeta$ remains unknown in this case,
in fact even the existence of a real number $\zeta$ 
with the property $\widehat{w}_{n}(\zeta)>n$
is open. The first result in this direction was due to Davenport and Schmidt~\cite{davsh},
who showed $\widehat{w}_{n}(\zeta)\leq 2n-1$ for any real $\zeta$.
Recently this bound has been refined in \cite{buschl} and further
in \cite{9},  
in the latter paper the currently best known bound
\begin{equation} \label{eq:buschlei}
\widehat{w}_{n}(\zeta)\leq \frac{3(n-1)+\sqrt{n^{2}-2n+5}}{2}
\end{equation}
was established.
The right hand side is of order $2n-2+o(1)$ as $n\to\infty$.   Conditionally on a conjecture of Schmidt and Summerer~\cite{sums},
small improvements of \eqref{eq:buschlei} can be obtained
with the method in~\cite{buschl}. In particular
for $n\geq 10$ it would imply $\widehat{w}_{n}(\zeta)\leq 2n-2$, see~\cite[Theorem~3.1]{iich}.

\section{The discrepancy between $w_{n}(\zeta)$ and $\widehat{w}_{n}(\zeta)$}

\subsection{The main result}
Investigating parametric geometry of numbers introduced by them,
Schmidt and Summerer~\cite{ssch} found the estimate 
\[
w_{n}(\zeta)\geq \frac{(n-1)\widehat{w}_{n}(\zeta)(\widehat{w}_{n}(\zeta)-1)}{1+(n-2)\widehat{w}_{n}(\zeta)},
\]
for the minimum discrepancy 
between $w_{n}(\zeta)$ and $\widehat{w}_{n}(\zeta)$, for any transcendental real $\zeta$. Rearrangements lead 
to the equivalent formulation
\begin{equation} \label{eq:lsws}
\widehat{w}_{n}(\zeta)\leq 
\frac{1}{2}\left(1+\frac{n-2}{n-1}w_{n}(\zeta)\right)+
\sqrt{\frac{1}{4}\left(\frac{n-2}{n-1}w_{n}(\zeta)+1\right)^{2}+\frac{w_{n}(\zeta)}{n-1}}.
\end{equation}
In fact analogous estimates were established in the more general context of linear forms with respect to any given
real vector $(\zeta_{1},\ldots,\zeta_{n})$ that is $\mathbb{Q}$-linearly independent together with $\{1\}$,
and in this case they are sharp for any dimension and parameter, see Roy~\cite{roy}.  
The above estimates yield a proper 
improvement of the obvious left inequality in \eqref{eq:wmono}
unless $w_{n}(\zeta)=\widehat{w}_{n}(\zeta)=n$. A non-trivial identity case occurs for $n=2$ and
$\zeta$ an extremal number as defined in Section~\ref{sek1}.
See also \cite{sums} for an improvement of \eqref{eq:lsws} when $n=3$, and a conjecture
concerning the optimal bound for arbitrary $n$.
A special case of the recent~\cite[Theorem~2.2]{buschl} complements \eqref{eq:lsws}.

\begin{theorem}    [Bugeaud, Schleischitz] \label{bugsch}
Let $n\geq 2$ be integers and $\zeta$ be a transcendental real number. Then in case of
\begin{equation} \label{eq:brack}
w_{n}(\zeta)> w_{n-1}(\zeta),
\end{equation}
we have
\begin{equation} \label{eq:kamel} 
\widehat{w}_{n}(\zeta)\leq \frac{nw_{n}(\zeta)}{w_{n}(\zeta)-n+1}.
\end{equation} 
\end{theorem}

Observe that in contrast to \eqref{eq:lsws}, the bound in \eqref{eq:kamel} for $\widehat{w}_{n}(\zeta)$ decreases
as $w_{n}(\zeta)$ increases. 
For $n=2$ and $\zeta$ any Sturmian continued fraction, see~\cite{buglau} for a definition, 
there is equality in \eqref{eq:kamel}. This can be verified by inserting the 
exact values of $w_{2}(\zeta)$ and $\widehat{w}_{2}(\zeta)$
determined in the main result of~\cite{buglau}. In particular, for extremal numbers mentioned above
we have equality in both \eqref{eq:lsws} and \eqref{eq:kamel} when $n=2$.

The condition \eqref{eq:brack} in Theorem~\ref{bugsch} was used predominately to guarantee
that the polynomials in the definition of $w_{n}$ have degree precisely $n$ (in the
special case of ~\cite[Theorem~2.2]{buschl} reproduced 
in Theorem~\ref{bugsch}). In other words,
for any $\epsilon>0$, there are arbitrarily large irreducible integer polynomials 
$P$ of degree exactly $n$ for which the estimate
\begin{equation} \label{eq:frisch}
\vert P(\zeta)\vert \leq H(P)^{-w_{n}(\zeta)+\epsilon}
\end{equation}
holds. This was a crucial observation for the proof.
We point out that this property does not hold in general, i.e. when we drop the condition \eqref{eq:brack}. 
Indeed, using continued fraction expansion, one can even construct real numbers for which
the degree of any $P$ which satisfies \eqref{eq:frisch} equals one, when $\epsilon$ is sufficiently small.
This can be deduced from the proof of~\cite[Corollary~1]{bug2010}. It is unknown whether \eqref{eq:kamel} still holds
when we drop the condition \eqref{eq:brack}.
The purpose of
this paper is to provide a weaker but unconditioned relation.
 We will agree on $w_{0}(\zeta)=0$ in our following main result. 

\begin{theorem} \label{nn}
Let $n$ be a positive integer and $\zeta$ be a real transcendental number. Let $l\in\{1,2,\ldots,n\}$ be the smallest integer such 
that $w_{l}(\zeta)= w_{n}(\zeta)$.
Then the estimation
\begin{equation} \label{eq:lhsrhs}
\widehat{w}_{n}(\zeta)\leq \min \left\{n+l-1, 
\frac{nw_{n}(\zeta)}{w_{n}(\zeta)-l+1}+w_{n-l}(\zeta)\cdot \left(1-\frac{n}{w_{n}(\zeta)-l+1}\right)\right\}
\end{equation}
holds.
\end{theorem}

\begin{remark} \label{rehmark}
In case of $w_{n}(\zeta)\leq n+l-1$, the trivial estimate in \eqref{eq:wmono} implies \eqref{eq:lhsrhs}.
More generally, when $w_{n}(\zeta)$ does not exceed
$n+l-1$ by much, the Schmidt-Summerer bound \eqref{eq:lsws} is even smaller than both bounds in \eqref{eq:lhsrhs}.
\end{remark}

The left bound in \eqref{eq:lhsrhs} will be an easy consequence of Theorem~\ref{lutz} from \cite{buschl}
reproduced below, the main new contribution is the right bound. When $l\leq \lfloor n/2\rfloor$, by definition of $l$ we have 
$w_{n-l}(\zeta)=w_{n}(\zeta)$ and the right
bound in the minimum in \eqref{eq:lhsrhs} becomes $w_{n}(\zeta)$, which is trivial in view of \eqref{eq:wmono}. Thus
Theorem~\ref{nn} is of interest primarily when $l>n/2$ 
and $w_{n}(\zeta)>n+l-1$.
However, if these relations hold and $w_{n-l}(\zeta)$ does not exceed $n-l$ by much,
then one checks that the right expression in the minimum in \eqref{eq:lhsrhs} is the smaller one.
In general, the new right bound in \eqref{eq:lhsrhs} is 
of interest when 
$l$ is rather close to $n$ and $w_{n-l}(\zeta)$ is relatively small, whereas $w_{n}(\zeta)$ is large.
For $l=n$, the bound in \eqref{eq:lhsrhs} becomes \eqref{eq:kamel},
and we recover Theorem~\ref{bugsch}. 
The expression $w_{n-l}(\zeta)$ involved in \eqref{eq:lhsrhs}
is unpleasant, as it can be arbitrarily close to $w_{n}(\zeta)$. We would like to replace it by $\widehat{w}_{n-l}(\zeta)$,
which could be effectively bounded with \eqref{eq:buschlei} by roughly $2(n-l)$. The proof will suggest
that such improvements are realistic. 

\subsection{$U_{m}$-numbers}
The claim of Theorem~\ref{nn} is of particular interest
when $\zeta$ is a $U_{m}$-number (see Section~\ref{sek1}).
In that case,
in~\cite[Corollary~2.5]{buschl} it was deduced essentially
from the generalization~\cite[Theorem~2.3]{buschl} of
Theorem~\ref{bugsch}
that $\widehat{w}_{m}(\zeta)=m$, and 
moreover
\begin{equation} \label{eq:effeff}
\widehat{w}_{n}(\zeta)\leq n+m-1,  \qquad\qquad n\geq 1.
\end{equation}
We remark that~\cite[Theorem~2.3]{buschl} rephrased in 
Theorem~\ref{lutz} below provides another proof of \eqref{eq:effeff}.
We can now refine this estimate when $n$ is roughly 
between $m$ and $2m$.

\begin{corollary} \label{jawasdenn}
Let $n>m\geq 1$ be integers and $\zeta$ be a $U_{m}$-number. Then
\[
\widehat{w}_{n}(\zeta)\leq n+ \min \left\{m-1, w_{n-m}(\zeta) \right\}.
\]
\end{corollary}

\begin{proof}
By assumption we have $w_{m-1}<w_{m}=w_{m+1}=\cdots=w_{n}=\infty$,
where we agree on $w_{0}(\zeta)=0$ if $m=1$.
Thus we may apply Theorem~\ref{nn} with $l=m$, which yields the
claimed bound.
\end{proof}

As indicated, the possible gain by the replacement of $m-1$ by $w_{n-m}(\zeta)$ in the minimum can only
occur when $n$ is not too large compared to $m$. More precisely $n<2m-1$ is a necessary condition by \eqref{eq:wmono}.
On the other hand, when $\zeta$ is a $U_{m}$-number and $w_{l}(\zeta)$ is small for 
some $l<m-1$, then Corollary~\ref{jawasdenn} yields a significant improvement for $n=m+l$. 
It is reasonable that even $U_{m}$-numbers with the property 
$w_{l}(\zeta)=l$ for $1\leq l\leq m-1$ exist. For $m=2$ this is true, using continued
fraction expansion one can even construct a $U_{2}$-number $\zeta$ with any prescribed value $w_{1}(\zeta)\in[1,\infty)$.
See~\cite[Theorem~7.6]{bugbuch} and its preceding remarks. However, for $U_{2}$-numbers we do not get any new
insight from Theorem~\ref{nn}. 
Concerning $U$-numbers of larger index, Alnia\c{c}ik~\cite{aln}
showed the existence of uncountably many $U_{m}$-numbers of arbitrary index $m\geq 2$ with the property $w_{1}(\zeta)=1$.
(In~\cite{alni2} the analogous result 
for $T$-numbers is proposed, 
however as pointed out by
Bugeaud in~\cite[Section~7.10]{bugbuch} crucial
estimates in~\cite{alni2} are not carried out properly
and serious revision of the paper is required.)          
For such $U_{m}$-numbers, the succeeding uniform 
exponent $\widehat{w}_{m+1}(\zeta)$ can
be bounded with Corollary~\ref{jawasdenn} as
\[
\widehat{w}_{m+1}(\zeta)\leq m+2.
\]
For large $m$, this leads to a reasonable improvement compared to the trivial bound $(m+1)+m-1=2m$ from \eqref{eq:effeff}. 
Moreover, Alnia\c{c}iks main theorem in~\cite{aln} seems to allow one to construct $U_{m}$-numbers  
with arbitrary prescribed value $w_{1}(\zeta)=w_{1}\in[1,\infty)$, thus extending the result        
for $U_{2}$-numbers above. Indeed, it suffices to take                                                  
$b_{s_{n}+1}=\nu_{s_{n}+1}=\lfloor q_{s_{n}+1}^{w_{1}-1}\rfloor$                                        
(in fact rather $b_{s_{n}+1}=\nu_{s_{n}+1}=\lfloor q_{s_{n}}^{w_{1}-1}\rfloor$ in the  
classical notation of continued fractions where $q_{n+1}=a_{n+1}q_{n}+q_{n-1}$ for $a_{j}$   
the partial quotients and $p_{n}/q_{n}$ the convergents, this seems to be a minor inaccuracy in~\cite{aln})  
and let the remaining (i.e. $j$ not of the form $s_{n}+1$) $b_{j}=a_{j}$ in the formulation of the theorem. 
Elementary facts on continued fractions and Roth's Theorem imply $w_{1}(\zeta)= w_{1}$.   
Strangely, this observation seems not to have been 
previously mentioned.                          
As soon as $w_{1}(\zeta)=w_{1}<m-1$, the resulting bound                                                 
$\widehat{w}_{m+1}(\zeta)\leq m+w_{1}+1$                                                                 
again improves the trivial upper bound $2m$.                                                             

On the other hand, the larger intermediate exponents $w_{2}(\zeta), w_{3}(\zeta),\ldots, w_{m-1}(\zeta)$ 
are hard to control for a $U_{m}$-number.
A construction of
Schmidt~\cite{schmidl} shows that it is possible to obtain
$w_{l}(\zeta)\leq m+l-1$ simultaneously for $1\leq l\leq m-1$
for some $U_{m}$-number $\zeta$. This refined
an earlier result of Alnia\c{c}ik, Avci and Bugeaud~\cite{alni}.   
However, this estimation is not sufficient to improve the previously known bound $\widehat{w}_{n}(\zeta)\leq m+n-1$
with Corollary~\ref{jawasdenn}.
Finally, we point out that $U_{m}$-numbers which allow arbitrarily good irreducible
polynomial evaluations $\vert P(\zeta)\vert$ of some degree $n>m$
as well satisfy $\widehat{w}_{n}(\zeta)=n$. By the condition
more precisely we mean that the exponent $w$ in \eqref{eq:w} can be chosen arbitrarily large among polynomials of degree $m$ and
additionally among irreducible polynomials $P$ of degree $n>m$.
Indeed, the second expression in the right hand estimate in \eqref{eq:lhsrhs} 
can be dropped in this case, as can be seen from the proof below.

\section{Proof of Theorem~\ref{nn}} \label{proofs}

We reproduce some results from~\cite{buschl} for the proof.
The first is~\cite[Theorem~2.3]{buschl}, which essentially implies the left bound in \eqref{eq:lhsrhs}.

\begin{theorem}[Bugeaud, Schleischitz]  \label{lutz}
Let $m,n$ be positive integers and $\zeta$ be a transcendental real number. Then 
\[
\min \{ w_{m}(\zeta), \widehat{w}_{n}(\zeta) \} \leq m+n-1.
\]
\end{theorem} 

Theorem~\ref{lutz} was recently refined~\cite{ichindag} by replacing the right hand side by
$1/\widehat{\lambda}_{m+n-1}(\zeta)$, for $\widehat{\lambda}_{m+n-1}(\zeta)\geq 1/(m+n-1)$ the classic 
exponent of uniform simultaneous rational approximation 
to $(\zeta,\zeta^{2},\ldots,\zeta^{m+n-1})$ 
for a real number $\zeta$. See for example~\cite{bug2010} for 
an exact definition.
We will further directly apply the following~\cite[Lemma~3.1]{buschl}, a generalization of~\cite[Lemma~8]{davsh}
by Davenport and Schmidt, which was the core of the proof 
of both Theorem~\ref{bugsch} and Theorem~\ref{lutz}. We write $A\ll_{.} B$ in the sequel when
$B$ exceeds $A$ at most by a constant that depends on the subscript variables. 

\begin{lemma}[Bugeaud, Schleischitz] \label{kastner}
Let $P, Q$ be two coprime integer polynomials of degree $m$ and $n$, respectively. 
Further let $\zeta$ be any real number. Then
at least one of the estimates
\[
\vert P(\zeta)\vert \gg_{m,n,\zeta} H(P)^{-n+1} H(Q)^{-m}, \qquad \vert Q(\zeta)\vert\gg_{m,n,\zeta} H(P)^{-n} H(Q)^{-m+1}
\]
holds. In particular
\[
\max \{ \vert P(\zeta)\vert, \vert Q(\zeta)\} \gg_{m,n,\zeta} H(P)^{-n+1} H(Q)^{-m+1} \min \{ H(P)^{-1},H(Q)^{-1}\}.
\]
\end{lemma}

In the formulation of the lemma we dropped the 
condition $\zeta P(\zeta)Q(\zeta)\neq 0$ stated in~\cite{buschl},
which is not required as pointed out to me by D. Roy. 
We further point out that Wirsing~\cite{wirsing} showed that for any $n\geq 1$
there exists a constant $K(n)>1$, such that uniformly for all polynomials $P,Q\in\mathbb{Z}[T]$ 
of degree at most $n$
\begin{equation} \label{eq:wiersing}
K(n)^{-1}H(P)H(Q) \leq H(PQ) \leq K(n)H(P)H(Q) 
\end{equation}
holds. He deduced that the polynomials within the definition of $w_{n}(\zeta)$ can be chosen irreducible.
Moreover it follows from the definition of the exponents that in the case of $w_{n}(\zeta)>w_{n-1}(\zeta)$
these irreducible polynomials have degree precisely $n$. This fact was already an essential ingredient in 
the proof of~\cite[Theorem~2.1]{buschl}, which is 
our Theorem~\ref{bugsch}.

\begin{proof} [Proof of Theorem~\ref{nn}]
First assume $w_{n}(\zeta)\leq l+n-1$. Then by \eqref{eq:wmono} clearly $\widehat{w}_{n}(\zeta)\leq l+n-1$ as well.
Moreover, it is easy to check that the right bound in \eqref{eq:lhsrhs} exceeds $n+l-1$, see also Remark~\ref{rehmark}. Hence we can restrict to $w_{n}(\zeta)>l+n-1$.
We will further assume $w_{n}(\zeta)<\infty$ for simplicity. The case $w_{n}(\zeta)=\infty$
can be treated very similarly by considering the polynomials $P$ for which $-\log \vert P(\zeta)\vert/\log H(P)$
tends to infinity.
 
By the choice of $l$ we have $w_{n}(\zeta)=w_{l}(\zeta)>w_{l-1}(\zeta)$. Hence,
as carried out above, for any $\epsilon>0$ there exist infinitely many
irreducible integer polynomials $P$ of degree 
precisely $l$ such that  
\begin{equation} \label{eq:beidesjo}
H(P)^{-w_{n}(\zeta)-\epsilon} \leq \vert P(\zeta)\vert \leq H(P)^{-w_{n}(\zeta)+\epsilon}.
\end{equation}
Fix one such irreducible $P$ of large height and small $\delta>0$ to be chosen later and let
\begin{equation} \label{eq:theta}
\theta= \frac{w_{n}(\zeta)-l+1}{n}, \qquad X_{\delta}= H(P)^{\theta-\delta}.
\end{equation}
We want to give a lower bound on $\vert Q(\zeta)\vert$ for $Q$ an arbitrary integer polynomial
of degree at most $n$ and height $H(Q)\leq X_{\delta}$. We distinguish two cases.

Case 1: The polynomial $Q$ is not a polynomial multiple of $P$.
Then $P,Q$ are coprime as $P$ is irreducible, and thus we may apply Lemma~\ref{kastner}.
First assume $\vert Q(\zeta)\vert \geq \vert P(\zeta)\vert$.
Then we infer
\begin{equation} \label{eq:portion}
-\frac{\log \vert Q(\zeta)\vert}{\log X_{\delta}}\leq -\frac{\log \vert P(\zeta)\vert}{\log X_{\delta}}
\leq \frac{w_{n}(\zeta)+\epsilon}{\theta-\delta}.
\end{equation}
The upper bound follows for such $Q$ as we may choose $\epsilon$ and $\delta$ arbitrarily small,
and doing so the right hand side in \eqref{eq:portion} tends to $nw_{n}(\zeta)/(w_{n}(\zeta)-l+1)$,
whereas the remaining expression in \eqref{eq:lhsrhs} is non-negative.
Now assume $\vert Q(\zeta)\vert < \vert P(\zeta)\vert$.
Then \eqref{eq:beidesjo} yields
\begin{equation} \label{eq:verallg}
\max \{ \vert P(\zeta)\vert, \vert Q(\zeta)\vert \}= \vert P(\zeta)\vert \leq H(P)^{-w_{n}(\zeta)+\epsilon}.
\end{equation}

First assume $H(Q)\leq H(P)$. Then Lemma~\ref{kastner} yields
\[
\max \{ \vert P(\zeta)\vert, \vert Q(\zeta)\vert \} \gg_{n,\zeta} H(P)^{-l} H(Q)^{-n+1}\geq H(P)^{-n-l+1}.
\]
This contradicts \eqref{eq:verallg} for sufficiently large $H(P)$ and sufficiently small $\epsilon>0$,
by our hypothesis $w_{n}(\zeta)>l+n-1$. Hence $H(Q)>H(P)$ must hold.
Then Lemma~\ref{kastner} implies
\[
\max \{ \vert P(\zeta)\vert, \vert Q(\zeta)\vert \} \gg_{n,\zeta} H(P)^{-l+1}H(Q)^{-n} 
\geq H(Q)^{-\frac{l-1}{\tau}} H(Q)^{-n},
\]
where $\tau= \log H(Q)/ \log H(P)>1$. Combination with \eqref{eq:verallg} yields
\[
\frac{w_{n}(\zeta)}{\tau}-\epsilon\leq 
\frac{w_{n}(\zeta)-\epsilon}{\tau}\leq n+\frac{l-1}{\tau},
\]
hence
\[
\tau \geq \frac{w_{n}(\zeta)-l+1}{n+\epsilon}=\theta\cdot \frac{n}{n+\epsilon}.
\]
This contradicts our assumption $H(Q)\leq X_{\delta}$, which is equivalent to $\tau\leq \theta-\delta$, when $\epsilon$ is 
chosen small enough compared to $\delta$. This contraction finishes the proof of case 1.

Case 2: The integer polynomial $Q$ is
of the form $Q=PV$ for some integer polynomial $V$. The 
degree of $V$ is at most $n-l$ since
$Q$ has degree at most $n$ and $P$ has degree precisely $l$. Moreover from Wirsing's estimate \eqref{eq:wiersing} we infer
\[
H(V)\ll_{n} \frac{H(Q)}{H(P)}\leq \frac{X_{\delta}}{H(P)}= H(P)^{\theta-1-\delta}.
\]
Let $\tilde{\varepsilon}>0$ be small. By definition of $w_{n-l}$, for $\varepsilon>0$ a variation of 
$\tilde{\varepsilon}$ (that tends to $0$ as $\tilde{\varepsilon}$ does) and for sufficiently large $H(P)$, 
all but finitely many $V$ satisfy
\begin{equation}  \label{eq:habsverkaesen}
\vert V(\zeta)\vert \geq H(V)^{-w_{n-l}(\zeta)-\tilde{\varepsilon}}\geq  H(P)^{-w_{n-l}(\zeta)(\theta-1-\delta)-\varepsilon}.
\end{equation}
We briefly discuss the possible exceptions 
$V\in\{V_{1},\ldots,V_{h}\}$ for the given $\tilde{\epsilon}$. 
By the finiteness and transcendence of $\zeta$ we infer an absolute lower bound $\max_{1\leq j\leq h} \vert V_{j}(\zeta)\vert \gg 1$. 
Thus for $V\in\{V_{1},\ldots,V_{h}\}$ we 
have $\vert Q(\zeta)\vert= \vert P(\zeta)\vert \cdot \vert V(\zeta)\vert \gg \vert P(\zeta)\vert$. 
The bound $nw_{n}(\zeta)/(w_{n}(\zeta)-l+1)$ follows similarly to \eqref{eq:portion} as $H(P)\to\infty$
and $\tilde{\epsilon}\to 0$. Now we treat the main case 
of $V$ that satisfy \eqref{eq:habsverkaesen}.
Together with \eqref{eq:beidesjo} and Wirsing's estimate we obtain
\[
\vert Q(\zeta)\vert = \vert P(\zeta)\vert\cdot \vert V(\zeta)\vert 
\geq H(P)^{-w_{n}(\zeta)-w_{n-l}(\zeta)(\theta-1-\delta)-(\epsilon+\varepsilon)}.
\]
We conclude
\[
-\frac{\log \vert Q(\zeta)\vert}{\log X_{\delta}} \leq 
\frac{w_{n}(\zeta)+w_{n-l}(\zeta)(\theta-1-\delta)}{\theta-\delta}+\frac{\epsilon+\varepsilon}{\theta-\delta}.
\]
As we may choose $\delta$ and $\epsilon, \varepsilon$ arbitrarily small, we obtain
\[
-\frac{\log \vert Q(\zeta)\vert}{\log X_{\delta}} \leq  
\frac{w_{n}(\zeta)+w_{n-l}(\zeta)(\theta-1)}{\theta}+\epsilon^{\prime},
\]
for arbitrarily small $\epsilon^{\prime}>0$. 
Inserting the value of $\theta$ from \eqref{eq:theta} we obtain
\[
-\frac{\log \vert Q(\zeta)\vert}{\log X_{\delta}} \leq 
\frac{nw_{n}(\zeta)}{w_{n}(\zeta)-l+1}+w_{n-l}(\zeta)\cdot \frac{w_{n}(\zeta)-n-l+1}{w_{n}(\zeta)-l+1}+\epsilon^{\prime}.
\]
The right bound in \eqref{eq:lhsrhs} follows with elementary rearrangements. 
Since this holds for any polynomial multiple $Q$ of $P$ of height $H(Q)\leq X_{\delta}$, 
the proof of case 2 is finished as well.
\end{proof}

 \vspace{1cm}
 
 The author thanks the anonymous 
 referee for the 
 careful reading.

\end{document}